\documentclass{amsart}
\usepackage{latexsym,amssymb,enumerate, amsmath}
\usepackage{color}

\newtheorem{theorem}{Theorem}[section]

\newtheorem{lemma}[theorem]{Lemma}
\newtheorem{proposition}[theorem]{Proposition}
\theoremstyle{definition}

\newtheorem{example}[theorem]{Example}

\newtheorem*{ThmA}{Theorem A}

\newenvironment{enumeratei}{\begin{enumerate}[\upshape (a)]}
    {\end{enumerate}}

\def\irr#1{{\rm Irr}(#1)}
\def\cent#1#2{{\bf C}_{#1}(#2)}

\def\syl#1#2{{\rm Syl}_#1(#2)}

\def\zent#1{{\bf Z}(#1)}

\newcommand{\N}{{\mathbb N}}
\newcommand{\F}{{\mathbb F}}

\def\irr#1{{\rm Irr}(#1)}

\def\cent#1#2{{\bf C}_{#1}(#2)}

\def\syl#1#2{{\rm Syl}_#1(#2)}

\def\zent#1{{\bf Z}(#1)}

\def\SL#1{{\rm SL}_{2}(#1)}
\def\PSL#1{{\rm PSL}_{2}(#1)}

\def\V#1{{\rm V}(#1)}

\def\irr#1{{\rm Irr}(#1)}

\def\cd#1{{\rm cd}(#1)}

\def\cent#1#2{{\bf C}_{#1}(#2)}
\def\syl#1#2{{\rm Syl}_#1(#2)}

\def\zent#1{{\bf Z}(#1)}

\mathchardef\coso="2023

\begin{document}

\title[Bounding the number of vertices of the degree graph]{Bounding the number of vertices \\in the degree graph of a finite group}

\author[Z. Akhlaghi et al.]{Zeinab Akhlaghi}
\address{Zeinab Akhlaghi, Faculty of Math. and Computer Sci., \newline Amirkabir University of Technology (Tehran Polytechnic), 15914 Tehran, Iran.}
\email{z\_akhlaghi@aut.ac.ir}

\author[]{Silvio Dolfi}
\address{Silvio Dolfi, Dipartimento di Matematica e Informatica U. Dini,\newline
Universit\`a degli Studi di Firenze, viale Morgagni 67/a,
50134 Firenze, Italy.}
\email{dolfi@math.unifi.it}

\author[]{Emanuele Pacifici}
\address{Emanuele Pacifici, Dipartimento di Matematica F. Enriques,
\newline Universit\`a degli Studi di Milano, via Saldini 50,
20133 Milano, Italy.}
\email{emanuele.pacifici@unimi.it}

\author[]{Lucia Sanus}
\address{Lucia Sanus, Departament de Matem\`atiques, Facultat de
 Matem\`atiques, \newline
Universitat de Val\`encia,
46100 Burjassot, Val\`encia, Spain.}
\email{lucia.sanus@uv.es}

\thanks{The research of the second and  third  author is partially supported by the Italian PRIN 2015TW9LSR\_006 ``Group Theory and Applications'.
The research of the fourth author is supported by  the Spanish  Ministerio de Economia y Competitividad proyecto MTM2016-76196-P,  partly with  FEDER funds, and PrometeoII/2015/011-Generalitat Valenciana}

\subjclass[2000]{20C15}

\begin{abstract}
Let \(G\) be a finite group, and let \(\cd G\) denote the set of degrees of the irreducible complex characters of \(G\). The \emph{degree graph} \(\Delta(G)\)  of \(G\) is defined as the simple undirected graph whose vertex set \(\V G\) consists of the prime divisors of the numbers in \(\cd G\), two distinct vertices \(p\) and \(q\) being adjacent if and only if \(pq\) divides some number in \(\cd G\). In this note, we provide an upper bound on the size of \(\V G\) in terms of the \emph{clique number} \(\omega(G)\) (i.e., the maximum size of a subset of \(\V G\) inducing a complete subgraph) of \(\Delta(G)\). Namely, we show that \(|\V G|\leq{\rm{max}}\{2\omega(G)+1,\;3\omega(G)-4\}\). 
Examples are given in order to show that the bound is best possible. This completes the analysis carried out in \cite{ACDKP} where the solvable case was treated, extends the results in \cite{AKT,ATV,TV}, and answers a question posed by the first author and H.P. Tong-Viet in \cite{ATV}. 
\end{abstract}

\maketitle

\section{Introduction}

The \emph{degree graph} of a finite group $G$, that we denote by $\Delta(G)$, is defined as the (simple undirected) prime graph
related to the set $\cd G$ of the irreducible complex character degrees of $G$. 
The  vertex
set \(\V G\) of $\Delta(G)$ consists of all the prime numbers that divide some $d \in \cd G$, and 
two distinct vertices $p$, $q$  are adjacent in $\Delta(G)$  if and only if there exists $d \in \cd G$ such that the product $pq$ divides $d$.

The properties of the graph $\Delta(G)$, as well as their connections with  the algebraic structure of the group $G$, 
have been object of investigation in the last few decades. We refer to~\cite{Lew} for a survey on this topic.

One question that has been considered,  concerns the existence of ``large'' complete subgraphs (cliques) in $\Delta(G)$.
Given a finite group \(G\), let us denote by \(\omega(G)\) the \emph{clique number} of \(\Delta(G)\), that is, the maximum size of a subset in \(\V G\) whose elements are pairwise adjacent in \(\Delta(G)\). H.P. Tong-Viet proves in~\cite{TV} that if $\omega(G) \leq 2$, then $|\V G| \leq 5$; moreover, he classifies the groups $G$ such that
  $\omega(G) = 2$ and $|\V G| = 5$. Again Tong-Viet together with the first author prove that, if \(G\) is a group with \(\omega(G)\leq 3\), then \(|\V G|\) is at most \(7\) (\cite{ATV}); this bound is best possible in general, but it is improved to \(6\) under the additional assumption that \(G\) is solvable. In view of these results, the authors of \cite{ATV} suggest that the inequality

\smallskip
\noindent \((1)\qquad\qquad\qquad\qquad\qquad\qquad |\V G|\leq 2\omega(G)+1\hfill\) 
\smallskip

\noindent should be true for any finite group \(G\), whereas, for \(G\) solvable, \(|\V G|\leq 2\omega(G)\) should hold (\cite[Conjecture 1]{ATV}). 

The latter inequality, for the solvable case, was proved true in \cite[Corollary~B]{ACDKP} as a consequence of the fact that the vertex set of $\Delta(G)$ is covered by two suitable subsets of \(\V G\) inducing complete subgraphs. We remark here that the same is not true without the solvability assumption; the smallest example is $\Delta({\rm PSL}_2(4))$, that has three vertices and no edges. As regards the general case, a recent paper (\cite{AKT}) provides some more evidence for Inequality (1), establishing that it holds also for \(\omega(G)= 4\). Nevertheless, (1) turns out to be false in general, as shown by the following example.

\begin{example}
\label{ex}
Let \(\Pi=\{u_1^{\alpha_1},...,u_n^{\alpha_n}\}\) be a set of prime powers where every prime \(u_i\) is larger than \(3\). Assume that, for every \(i\in\{1,...,n\}\), we have \(|\pi(u_i^{\alpha_i}-1)\setminus\{2,3\}|=|\pi(u_i^{\alpha_i}+1)\setminus\{2,3\}|=1\), and assume further that, for distinct \(r\) and \(s\) in \(\{1,...,n\}\), the intersection of the sets \(\{u_r\}\cup\pi(u_r^{2\alpha_r}-1)\) and \(\{u_s\}\cup\pi(u_s^{2\alpha_s}-1)\) is \(\{2,3\}\). 
Now, setting \(G_{\Pi}=\PSL{u_1^{\alpha_1}}\times\cdots\times\PSL{u_n^{\alpha_n}}\), it is easy to see that \(|\V {G_{\Pi}}|=3n+2\) and \(\omega(G_{\Pi})=n+2\), thus \(|\V {G_{\Pi}}|=2\omega(G_{\Pi})+n-2=3\omega(G_{\Pi})-4\). As a consequence, if \(n\geq 4\), Inequality (1) does not hold for the group \(G_{\Pi}\).


Note that \(\Pi=\{29, 67, 157, 227\}\) is a set of four prime powers (in fact, of four primes) satisfying the above conditions. This provides a counterexample to Inequality (1) in which the number of vertices is \(14\), whereas the largest clique has size \(6\).
\end{example}

We remark that, if it is possible to find sets \(\Pi_n\) as in Example~\ref{ex} \emph{having arbitrarily large size \(n\)}, then the ratio \(|\V{G_{\Pi_n}}|/\omega(G_{\Pi_n})\) would converge to \(3\) as \(n\) tends to infinity (for instance, in the set of all prime numbers up to \(10^6\), there exists a subset \(\Pi\) satisfying the relevant conditions, with \(|\Pi|=15615\); for the group \(G_{\Pi}\) constructed as in Example~\ref{ex}, we thus have that \(|\V{G_{\Pi}}|/\omega(G_{\Pi})\) is larger than \(2.999\)); this leads us to conjecture that \emph{for every positive real number \(\epsilon\), there exists a group \(G\) such that \(|\V{G}|/\omega(G)>3-\epsilon\).} In fact, in \cite{EG}, the authors estimate the asymptotic density of one particular set of primes \(\overline{B}\), whose infinitude yields the existence of a set \(\Pi_n\) as above for every \(n\in\N\); a consequence of their work is that \(\overline{B}\) is actually infinite, assuming a generalized form of the Hardy-Littlewood conjecture. 

\smallskip
On the other hand, as the main result of this note, we prove:

\begin{ThmA}
Let \(G\) be a finite group. Then the following conclusions hold.
\begin{enumeratei}
\item If \(\omega(G)\leq 5\), then  $|\V G|\leq 2\omega(G)+1$.
\item If \(\omega(G)\geq 5\), then \(|\V G|\leq 3\omega(G)-4\).
\end{enumeratei}
\end{ThmA}

\noindent(In other words, for every finite group \(G\), we have that $|\V G|$ is bounded above by the largest number among $2\omega(G)+1$ and $3\omega(G)-4$; note that these two numbers match at the value \(11\) for \(\omega(G)=5\).) In view of Example~\ref{ex} and of the comments following it, this bound is best possible. Also, our proof of Theorem~A(a) provides a shorter argument for the theorems of \cite{AKT,ATV,TV}, in which Inequality~(1) is established for \(\omega(G)\leq 4\), and extends that analysis by showing that (1) holds for \(\omega(G)=5\) as well.

A key ingredient in our proof of Theorem~A is the main theorem of \cite{ACDPS}, which provides some detailed structural information concerning finite groups \(G\) such that \(|\V G|\leq 2\omega(G)\) does not hold. The part of this information that is relevant for our purposes is gathered in Proposition~\ref{applications1}.  

Finally, in the following discussion every group is assumed to be finite, and the classification of finite simple groups is involved (via the main theorem of \cite{ACDPS}).

\section{The results}

We start by recalling some facts concerning the degree graph of non-solvable groups. 

\begin{lemma}\label{PSL2}
Let $S \simeq \PSL{u^{\alpha}}$ or $S \simeq \SL{u^{\alpha}}$, where $u$ is a prime and $\alpha \geq 1$. 
Let $\pi_{+} = \pi(u^{\alpha}+1)$ and $\pi_{-} = \pi(u^{\alpha}-1)$. For a subset $\pi$ of vertices 
of $\Delta(S)$, we denote by $\Delta_{\pi}$ the subgraph of $\Delta = \Delta(S)$ induced
by the subset $\pi$.
\begin{enumeratei}
\item If $u=2$, then $\Delta(S)$ has three connected components, $\{u\}$, $\Delta_{\pi_{+}}$ and 
$\Delta_{\pi_{-}}$, and each of them is a complete graph.  
\item If $u > 2$ and $u^{\alpha} > 5$, then  $\Delta(S)$ has two connected components, 
$\{u\}$ and  $\Delta_{\pi_{+} \cup \pi_{-}}$; also, both  $\Delta_{\pi_{+}}$ and $\Delta_{\pi_{-}}$ are
complete graphs, no vertex in $\pi_{+}\setminus\{2\}$ is adjacent to any vertex in  
$\pi_{-}\setminus\{2\}$, and $2$ is adjacent to all other vertices in $\Delta_{\pi_{+} \cup \pi_{-}}$. 
\end{enumeratei}
\end{lemma}

\begin{proof}
It is well known (see for instance~\cite[Theorems 38.1, 38.2]{Dor}) that, for $\alpha \geq 2$,
  $$\cd{\SL{2^{\alpha}}} = \cd{\PSL{2^{\alpha}}} = \{1, 2^{\alpha} -1, 2^{\alpha}, 2^{\alpha} +1\}, $$
  and that, for $u \neq 2$  and $u^{\alpha} > 5$,
  $$ \cd{\PSL{u^{\alpha}}} = \{1, u^{\alpha} -1, u^{\alpha} , u^{\alpha}  +1, \frac{1}{2}(u^{\alpha}  + \epsilon)\} \text{ where }
  \epsilon = (-1)^{\frac{u^{\alpha} -1}{2}},$$
  $$\cd{\SL{u^{\alpha} }} = \{1, u^{\alpha}  -1, u^{\alpha} , u^{\alpha}  +1, \frac{1}{2}(u^{\alpha}  + \epsilon)\} \text{ where }
  \epsilon = \pm 1$$
(while $\cd{\PSL{5}} = \{1, 3, 4, 5\}$ and $\cd{\SL{5}} = \{1, 2, 3, 4, 5, 6\}$).
\end{proof}

\begin{lemma}\label{W}
  Let $G$ be an almost-simple group with socle $S \simeq \PSL{u^{\alpha}}$, where \(u\) is a prime. Let \(s\neq u\) be a prime divisor of $|G/S|$; then the following conclusions hold.
  \begin{enumeratei}
  \item  The prime \(s\) is adjacent in \(\Delta(G)\) to every prime in \(\pi(u^{2\alpha}-1)\).
    \item The prime $s$ is adjacent in $\Delta(G)$ to every prime in $\pi(G) \setminus \pi(S)$.
    \item The set of vertices   $\V G \setminus \{ u\}$ is covered by two complete subgraphs of $\Delta(G)$.
\end{enumeratei}
\end{lemma}

\begin{proof}
 Part (a)  follows from Theorem~A of~\cite{W2}.  Brauer Permutation Lemma, with~\cite[Proposition 2.6]{MT} and~\cite[Lemma 2.10]{MT}, yields (b).
  Finally,  (c) is a consequence of (a), (b) and Lemma~\ref{PSL2}.
\end{proof}
  
\begin{lemma}
  \label{CS}
  Let $G$ be a group, $M$ a non-abelian minimal normal subgroup of $G$ and $C = \cent GM$.
  Then the following conclusions hold.
  \begin{enumeratei}
  \item If $q$ is a prime divisor of $|G/MC|$ and $q$ does not divide $|M|$, then there exists $\theta \in \irr M$ such that $q$ divides $|G:I_G(\theta)|$.
  \item If $q$ is a prime divisor of $|G/C|$, then there exists $\theta \in \irr M$ such that
    $q$ divides $\chi(1)$ for all $\chi \in \irr{G|\theta}$.
  \item If $M$ is not a simple group, then $\Delta(G/C)$ is a complete graph. 
  \end{enumeratei}
\end{lemma}

\begin{proof}
  See~\cite[Proposition 2.6]{ACDPS}.
\end{proof}

As mentioned in the Introduction, the next result (which depends on Theorem~A of \cite{ACDPS}) will be crucial in the proof of the main result of this paper.
In the following statement, by \({\overline\Delta}(G)\) we denote the \emph{complement} of the graph \(\Delta(G)\): this is the graph whose vertex set is the same as of \(\Delta(G)\), and two vertices are adjacent in \({\overline\Delta}(G)\) if and only if they are not adjacent in \(\Delta(G)\).
We list here some key properties of the groups whose complement degree graph $\overline{\Delta}(G)$ is not bipartite.


\begin{proposition}
\label{applications1}
Let \(G\) be a group.
\begin{enumeratei}
\item Assume that there exists $\pi \subseteq \V G$, with $|\pi|$ an odd number larger than \(1\), such that $\pi$ is the set of vertices of a cycle in ${\overline{\Delta}}(G)$. Then there exists a characteristic subgroup \(N\) of \(G\), with \(\pi\subseteq\pi(N)\), such that $N$ is a $2$-dimensional special or projective special linear group over a finite field \(\F\), with \(\F\geq 4\). 
\item Let \(N\trianglelefteq G\) be such that $N$ is isomorphic either to \(\PSL{u^{\alpha}}\) or to  \(\SL{u^{\alpha}}\), where \(u^{\alpha}\geq 4\) is a prime power. Then \(|\cent G N \cap N|\leq 2\), and the prime divisors of $|G/N\cent G N|$ are adjacent in \(\Delta(G)\) to all primes in $\V N \setminus \{u\}$. 
\item Let \({\mathcal K}\) be any (non-empty) set of normal subgroups of \(G\) as in {\rm(b)} (with possibly different values of \(u^{\alpha}\)), and define \(K\) as the product of all the subgroups in \({\mathcal K}\); also, set \(C=\cent G K\). Then every prime \(t\) in \(\V C\) is adjacent in \(\Delta(G)\) to all the primes \(q\) (different from \(t\)) in \(|G/C|\), with the possible exception of \((t,q)=(2,u)\) when \(|{\mathcal K}|=1\), $K \simeq \SL {u^{\alpha}}$  for some $u \neq 2$ and
  $\zent K = P'$ for $P \in \syl 2C$. 
In any case, we have
  \[\V G=\V{G/C}\cup\V C . \]
\end{enumeratei}
\end{proposition}
\begin{proof}
 Parts (a) and (b) follow immediately from \cite[Theorem~A]{ACDPS} and Lemma~\ref{W}(a), respectively. 

As for Part (c), let \(t\) be a vertex of \(\Delta(C)\), and let \(q\neq t\) be in \(\pi(G/C)\). Since \(C=\bigcap\cent G N\) where \(N\) runs in \({\mathcal K}\), there exists \(N_1\in{\mathcal K}\) such that \(q\) is a divisor of \(|G/\cent G {N_1}|\). Moreover, \(C\) is a normal subgroup of \(C_1=\cent G {N_1}\), thus \(t\) is certainly a vertex of \(\Delta(C_1)\) as well. Set \(Z_1=\zent{N_1}\), and assume for the moment that \(t\) is in fact a vertex of \(\Delta(C_1/Z_1)\) (which is easily seen to be true whenever \(t\) is odd, as \(Z_1\) is a \(2\)-group). In this situation, let \(\phi\in\irr{C_1/Z_1}\) be such that \(\phi(1)\) is divisible by \(t\). If \(q\) is in \(\pi(N_1C_1/C_1)\), then we can consider \(\theta\in\irr{N_1/Z_1}\) whose degree is divisible by \(q\), and every \(\chi\in\irr{G\mid\phi\cdot\theta}\) will be such that \(qt\) divides \(\chi(1)\). On the other hand, if \(q\) is not in \(\pi(N_1C_1/C_1)\), then we can choose \(\theta\in\irr{N_1/Z_1}\) as in Lemma~\ref{CS}(a) and, as above, any  \(\chi\in\irr{G\mid\phi\cdot\theta}\) will do. 

Let us now assume that \(t\in\V{C_1}\setminus\V{C_1/Z_1}\); then, as observed, we have \(t=2\), and clearly also \(Z_1\neq 1\) (which implies that the characteristic \(u_1\) of \(N_1\) is odd). In this case, we see (using Lemma~\ref{PSL2} and Lemma~\ref{W}) that \(t\) and \(q\) are already adjacent in \(\Delta(G/C_1)\) unless possibly when \(q=u_1\). However, if \(\{N_1,...,N_{\ell}\}\) is a subset of \({\mathcal K}\) of the smallest possible size such that \(K=N_1\cdot N_2 \cdots N_{\ell}\), then it is easily seen that \(K/\zent K\) is isomorphic to the direct product \(N_1/\zent{N_1}\times\cdots\times N_{\ell}/\zent{N_{\ell}}\); we deduce that \(2\) is certainly adjacent in \(\Delta(K/\zent K)\) (thus in \(\Delta(G)\)) to all odd prime divisors of \(|K|\), including \(u_1\),  unless \(\ell=1\). Note that, in this situation, we have \(K=N_1\), \(C=C_1\), and \(\zent K=Z_1\).
If $P$ is a Sylow $2$-subgroup of $C$, then $|P'| = 2$ as $2 \not \in \V{C/\zent K}$ and $|\zent K| = 2$.
Since $2 \in \V C$, it follows that $\zent K = P'$, as wanted.  

Finally, we consider the last claim of Part (c). Let \(t\) be in \(\V G\). If \(t\) is a divisor of \(|G/C|\), then, as above, \(t\) divides \(|G/\cent G {N_1}|\) for some \(N_1\in{\mathcal K}\); as a consequence, \(t\) lies in \(\V{G/\cent G {N_1}}\subseteq\V{G/C}\). But if \(t\) is not a divisor of \(|G/C|\), then a Sylow \(t\)-subgroup \(T\) of \(G\) lies in \(C\); this \(T\) cannot be abelian and normal in \(C\) (as otherwise it would be normal in \(G\), contradicting the fact that \(t\) is a vertex of \(\Delta(G)\)), thus \(t\) lies in \(\V C\).
\end{proof}

We remark that the exception mentioned in part (c) of Proposition~\ref{applications1}
does in fact occur. Namely, if $G$ is the central product of $\SL 5$ and the
quaternion group $Q_8$, then $\cd G = \{  1, 3, 4, 5, 8, 12\}$, so $2 \in \V C$,
where $C = Q_8$, is not adjacent to the vertex  $5$ in $\Delta(G)$.  

We can now prove Theorem~A, that we state again.

\begin{ThmA}
\label{ATV}
Let \(G\) be a group. Then the following conclusions hold.
\begin{enumeratei}
\item If \(\omega(G)\leq 5\), then  $|\V G|\leq 2\omega(G)+1$.
\item If \(\omega(G)\geq 5\), then \(|\V G|\leq 3\omega(G)-4\).
\end{enumeratei}
\end{ThmA}

\begin{proof}[Proof of {\rm (a)}] If  $\overline{\Delta}(G)$ has no cycles of odd length, then it is a bipartite graph, so there are two cliques in \(\Delta(G)\) which cover together all the elements in \(\V G\) and we are done.

  We can therefore assume that there are cycles of odd length in $\overline{\Delta}(G)$.  Thus, by Proposition~\ref{applications1}(a), the set \({\mathcal K}\) of all normal subgroups of \(G\) that are $2$-dimensional special or projective special linear groups over a finite field \(\F\) (with \(\F\geq 4\)) is non-empty, and we can define \(K\) as the product of all the subgroups in \({\mathcal K}\). Let us fix a subset \(\{N_1,...,N_\ell\}\) of \({\mathcal K}\)
such that \(K/\zent K\simeq N_1/\zent{N_1}\times\cdots\times N_\ell/\zent{N_\ell}\).
For \(i\in\{1,...,\ell\}\), set \(C_i=\cent G{N_i}\), and let \(u_i^{\alpha_i}\) be a prime power such that \(N_i/\zent{N_i}\simeq\PSL{u_i^{\alpha_i}}\). Also, set \(C=\cent G K\).

We proceed by induction on the order of the group.  
First, we reduce to the case   $C=1$.
As obviously $C$ does not have any characteristic subgroup isomorphic to a \(2\)-dimensional special or projective special linear group over a finite field with at least four elements, Proposition~\ref{applications1}(a) ensures that \({\overline{\Delta}}(C)\) is a bipartite graph, whence there exist two cliques in \(\Delta(C)\) which cover all the elements of \(\V C\); in particular, $|\V C \setminus \V{G/C}|\leq 2\omega$, where $\omega$ is the maximum size of a clique in the subgraph of \(\Delta(C)\) induced by $\V C \setminus \V{G/C}$.
Hence, observing that \(\omega(G/C)\leq 5\), assuming $C\neq 1$  yields by  induction  $|\V{G/C}|\leq 2\omega(G/C)+1$, and we deduce that \(|\V G|\leq 2(\omega+\omega(G/C))+1\leq 2\omega(G)+1\) by an application of Proposition~\ref{applications1}(c) (observe that the possible exception discussed in Proposition~\ref{applications1}(c) is not relevant here; in fact, when \(K=N_1\) and \(u_1\neq 2\), a clique of \(\Delta(G/C)\) having maximal size does not involve \(u_1\)).
Thus, we assume that \(C=1\); then, in particular, \(\zent K\) is trivial and so are the subgroups \(\zent{N_i}\) for all \(i\in\{1,...,\ell\}\).

In view of the last observation in the above paragraph, Proposition~\ref{applications1}(c) yields that all the primes in $\V{C_i}$ are adjacent  to all the prime divisors of the almost simple group (with socle $\PSL{u_i^{\alpha_i}}$) $G/C_i$. 
We can  assume that  $\ell\geq 2$, as otherwise \(G\) is an almost simple group whose socle is isomorphic to \(\PSL{u^{\alpha}}\) for some \(u^{\alpha}\geq 4\), and  Lemma~\ref{W}(c)  yields $|\V G|\leq 2 \omega(G)+1$.

We  claim  that we can reduce to the situation when, for every \(i\in\{1,...,\ell\}\), we have:  
\begin{enumerate}
\item \(u_i\not\in\V{C_i}\), and
\item neither \(\pi(u_i^{\alpha_i}-1)\) nor \(\pi(u_i^{\alpha_i}+1)\) is contained in \(\V{C_i}\).
\end{enumerate}
In fact, as \(\omega(C_i)\) is clearly at most \(5\), induction yields \(|\V{C_i}|\leq 2\omega(C_i)+1\); moreover, by Lemma~\ref{W}(c) three cliques of \(\Delta(G/C_i)\) are enough to cover all the vertices in \(\V{G/C_i}\). But if we assume the contrary of either (1) or (2), then in fact two cliques of \(\Delta(G/C_i)\) cover all the vertices in \(\V{G/C_i}\setminus\V{C_i}\); thus, if \(\omega\) is the clique number of the subgraph of \(\Delta(G/C_i)\) induced by \(\V{G/C_i}\setminus\V{C_i}\), we get \(|\V{G/C_i}\setminus\V{C_i}|\leq 2\omega\).
Therefore,
we obtain \[|\V{G}|=|\V{G/C_i}\cup\V {C_i}|\leq 2(\omega(C_i)+\omega)+1\leq 2\omega(G)+1,\] and we are done.

  As a  consequence,  the \(N_i\) have pairwise distinct characteristics (by (1)), that is $u_i \neq u_j$ for $1 \leq i \neq j \leq\ell$,  and $u_i \neq 2, 3$ for all $1 \leq i \leq\ell$.

Assume then $\ell=2$. Then both $C_1$ and $G/C_1$ are almost-simple groups (with socle isomorphic, respectively, to $\PSL{u_2^{\alpha_2}}$  and $\PSL{u_1^{\alpha_1}}$).  
    By Lemma~\ref{W}, the primes in \(\V{C_1}\setminus\{u_2\}\) are covered by two cliques of \(\Delta(C_1)\) \emph{which intersect at least in the vertex \(2\)}, so \(|\V{C_1}|\leq 2\omega(C_1)\).
    On the other hand, also the primes in \(\V{G/C_1}\setminus \{u_1\}\) are covered by two cliques in $\Delta(G/C_1)$ (and hence in $\Delta(G)$). Hence, it follows that   \(|\V{G/C_1}\setminus \V{C_1}|\leq 2\omega_0 +1\) where \(\omega_0\) is the clique number of the subgraph of \(\Delta(G/C_1)\) induced by the set of vertices  \(\V{G/C_1}\setminus \V{C_1}\).
    As all the primes in $\V{C_1}$ are adjacent  to all the primes in $\V {G/C_1}$, we get $|\V G|\leq 2\omega(G)+1$.

    We can hence assume that \(\ell\) is at least \(3\).
    As $u_i \neq 2$,  for all $i \in \{1, \ldots, \ell\}$ we have that $\{2,3\}$ lies in \(\pi_i^*\), where $\pi_i^*$ is either $\pi(u_i^{\alpha_i} +1)$ or $\pi(u_i^{\alpha_i} -1)$.
    Note that \(C_j\) contains \(N_i\) for all \(j\neq i\), so   $\{2,3\}$ lies in \(\pi(C_j)\) for every \(j\in\{1,...,\ell\}\).
    Recalling (2), we deduce that there exists a prime $p_i \in \pi_i^*$, $p_i \neq 2,3$,   and hence  $\omega(N_i) \geq 3$.
    If  a set $V_1$ of four primes in (say)  $\V{N_1}$ induces a clique in $\Delta(N_1)$, then  the set \(V_1 \cup \{u_2,u_3\}\) would induce a clique in \(\Delta(G)\), a contradiction;
    the same of course holds for any  \(N_i\).
Therefore, we have $\omega(N_i)=3$ for all $i \in \{1, \ldots, \ell\}$, and  $\pi_i^* = \{2, 3, p_i\}$.
    Arguing along the same line, we deduce that $\ell=3$ and  $\omega(G)=5$.

Also,  only the clique of \(\Delta(N_1)\) containing \(\{2,3\}\) can have three vertices. In fact, if (say)  \(\{2,q_1,s_1\}\) induces a clique in \(\Delta(N_1)\) with \(3\not\in\{q_1,s_1\}\), we would have $\{q_1,s_1, 2, 3, u_2,u_3\}\) inducing a clique in \(\Delta(G)\); the same of course holds for \(N_2\) and \(N_3\).
Recalling (2), a clear consequence of this fact is that \(|\V{N_i}|=5\) for \(i\in\{1,2,3\}\); moreover, we have \(\V{N_i}\cap\V{N_j}=\{2,3\}\) whenever \(i\neq j\). 

Finally, we have $\V{G/C_i} = \V{N_i}$ for $i\in \{1,2,3\}$ as well. Namely, if (say) \(\V{G/C_1}\) has an element \(r\) which is not in \(\V{N_1}\), then
\(\{2,3,p_1,r,u_2,u_3\}\) would induce a clique in \(\Delta(G)\).

As a consequence,  we have $\V G=\V{N_1\times N_2\times N_3}$ and hence by induction  we may assume $G=N_1\times N_2\times N_3$. 
This implies that $|\V G|=5+3+3=11 = 2 \omega(G) +1$, and the proof is complete.
\end{proof}

\begin{proof}[Proof of {\rm (b)}]
Let \(G\) be a counterexample of minimal order to the statement; as the conclusions in (a) and (b) coincide for \(\omega(G)=5\), we have \(\omega(G)\geq 6\). Moreover, as in (a), the graph \({\overline\Delta}(G)\) must have cycles of odd length, whence there exists a subgroup \(N\) of \(G\) as in the conclusion of Proposition~\ref{applications1}(a). 

Set \(C=\cent G N\), and define \(\omega\) to be the maximum size of a clique in the subgraph of \(\Delta(G/C)\) induced by \(\V{G/C}\setminus\V C\); recalling Lemma~\ref{PSL2} and Proposition~\ref{W}, we see that \(|\V{G/C}\setminus\V C|\leq 2\omega+1\). Moreover, we have \(|\V C|\leq 2\omega(C)+1\) if \(\omega(C)\leq 5\) and, by the minimality of \(G\), \(|\V C|\leq 3\omega(C)-4\) if \(\omega(C)\geq 5\). Note that \(\omega\) cannot be \(0\), as otherwise (by Proposition~\ref{applications1}(c)) we would have \(\V G=\V C\), and the previous inequalities would produce an immediate contradiction.

Assuming that we are in the case \(\omega(C)\leq 5\), Proposition~\ref{applications1}(c) yields $|\V G|\leq 2(\omega+\omega(C))+2$. Now, if \(\omega+\omega(C)\leq 5\), we get \(|\V G|\leq 12<14\leq 3\omega(G)-4\), and \(G\) is not a counterexample; on the other hand, if \(\omega+\omega(C)\geq 6\), then \(\V G\leq 2(\omega+\omega(C))+2\leq 3(\omega+\omega(C))-4\leq 3\omega(G)-4\), again a contradiction.

Thus, we must have \(\omega(C)\geq 5\). But then we get $$|\V G|\leq(2\omega+1)+ (3\omega(C)-4)\leq 3(\omega+\omega(C))-4\leq 3\omega(G)-4,$$ the final contradiction that completes the proof.
\end{proof}

\section*{Acknowledgements}
This research has been carried out during a visit of the first and fourth author at the Dipartimento di Matematica e Informatica ``U. Dini" (DIMAI) of the University of Firenze.  They wish to thank the DIMAI for the hospitality.


\begin{thebibliography}{99}
 
\bibitem{ACDKP} Z. Akhlaghi, C. Casolo, S. Dolfi, K. Khedri, E. Pacifici, \emph{On the character degree graph of solvable groups}, Proc. Amer. Math. Soc. 146 (2018), 1505--1513.

\bibitem{ACDPS} Z. Akhlaghi, C. Casolo, S. Dolfi, E. Pacifici, L. Sanus, \emph{On the character degree graph of finite groups}, submitted, arXiv:1809.10415.

\bibitem{AKT} Z. Akhlaghi, K. Khedri, B. Taeri, \emph{Finite groups with $K_5$-free prime graphs}, Comm. Algebra, to appear.

\bibitem{ATV} Z. Akhlaghi, H.P. Tong-Viet, \emph{Finite groups with $K_4$-free prime graphs}, Algebr. Represent. Theor. 18 (2015), 235--256.




\bibitem{Dor} L. Dornhoff, \emph{Group Representation Theory}, M. Dekker, New York,  1971. 


\bibitem{EG} C. Esparza, L. Gehring, \emph{Estimating the size of a set of primes with applications to Group Theory}, submitted, arXiv:1810.08679. 




\bibitem{Lew} M.L. Lewis, \emph{An overview of graphs associated with character degrees and conjugacy class sizes in finite groups}, Rocky Mountain J. Math. 38 (2008), 
175--211.





\bibitem{MT} A. Moret\'o, P.H. Tiep, \emph{Prime divisors of character degrees}, J. Group Theory 11 (2008), 341--356.



\bibitem{TV} H.P. Tong-Viet, \emph{Groups whose prime graphs have no triangles}, J. Algebra, 378 (2012), 196-206.

\bibitem{W2} D.L. White, \emph{Character degrees of extensions of $\PSL{q}$ and $\SL{q}$}, J. Group Theory 16 (2013), 1--33. 





\end{thebibliography}
\end{document}